\documentclass[11pt]{amsart}

\usepackage{amsmath,amssymb}

\newtheorem{theorem}{Theorem}[section]
\newtheorem{lemma}[theorem]{Lemma}
\newtheorem{proposition}[theorem]{Proposition}
\newtheorem{corollary}[theorem]{Corollary}

\theoremstyle{definition}

\theoremstyle{remark}

\title[A truncated singular-value bound for spectral variation]
{A Truncated Singular-Value Bound for Spectral Variation of Normal Matrices}
\author{Qiyue Tang}

\address{Zhili College, Tsinghua University, Beijing, 100084, China}
\email{tangqy24@mails.tsinghua.edu.cn}


\newcommand{\diag}{\operatorname{diag}}
\newcommand{\C}{\mathbb{C}}
\newcommand{\KFnorm}[2]{\lVert #1\rVert_{(2,#2)}}

\begin{document}

\begin{abstract}
Let $A,B\in M_n(\C)$ be normal and let
$d(\sigma(A),\sigma(B))$ denote the optimal matching distance
between their spectra.  We prove that
\[
d(\sigma(A),\sigma(B))^2
\le
\sum_{k=1}^{\lfloor(n+1)/2\rfloor}s_k(A-B)^2,
\]
where $s_1(M)\ge\cdots\ge s_n(M)\ge0$ are the singular values of $M$.
Consequently,
\[
c_n\le \sqrt{\left\lfloor\frac{n+1}{2}\right\rfloor},
\]
improving the classical dimension-free bound $c_n\le c_*<2.9038872828$
for $3\le n\le16$;
in particular, $c_3\le\sqrt2$.  The truncation length is optimal for
$n=3$ and $n=4$, although the sharp values of the fixed-dimensional
constants remain unknown.  It also strengthens the Hoffman--Wielandt
bound by replacing the full Frobenius norm with a truncated
singular-value norm.  The proof
combines Hall's marriage theorem with intersecting spectral subspaces and
singular-value monotonicity under rectangular compressions.
\end{abstract}

\maketitle

\section{Introduction}

Let $A,B\in M_n(\C)$ be normal matrices, with spectra
\[
\sigma(A)=\{\alpha_1,\ldots,\alpha_n\},\qquad
\sigma(B)=\{\beta_1,\ldots,\beta_n\},
\]
where eigenvalues are counted with algebraic multiplicity.  Their
\emph{optimal matching distance} is
\[
d(\sigma(A),\sigma(B))
=\min_{\tau\in S_n}\max_{1\le i\le n}
  |\alpha_i-\beta_{\tau(i)}|.
\]
Throughout, $\lVert\cdot\rVert$ denotes the spectral norm,
$\lVert\cdot\rVert_F$ the Frobenius norm, and $M^*$ the conjugate
transpose (Hermitian adjoint) of $M$.  For a matrix $M$, let
\[
s_1(M)\ge s_2(M)\ge\cdots\ge s_n(M)\ge0
\]
be its singular values.  Thus $s_1(M)=\lVert M\rVert$ and
\[
\sum_{k=1}^n s_k(M)^2=\lVert M\rVert_F^2.
\]
For $1\le r\le n$, write
\[
\KFnorm{M}{r}:=
\left(\sum_{k=1}^r s_k(M)^2\right)^{1/2}
\]
for the Ky Fan $(2,r)$-norm, a unitarily invariant norm satisfying
$\KFnorm{M}{1}=\lVert M\rVert$ and $\KFnorm{M}{n}=\lVert M\rVert_F$.

The Hoffman--Wielandt theorem \cite{HoffmanWielandt1953} states that for
normal matrices there is a permutation $\tau\in S_n$ such that
\[
\sum_{i=1}^n|\alpha_i-\beta_{\tau(i)}|^2
\le \lVert A-B\rVert_F^2.
\]
It follows that
\[
d(\sigma(A),\sigma(B))\le \lVert A-B\rVert_F
\le\sqrt n\,\lVert A-B\rVert.
\]
Mirsky \cite{Mirsky1960} conjectured the stronger operator-norm estimate
\[
d(\sigma(A),\sigma(B))\le\lVert A-B\rVert
\]
for all normal $A$ and $B$.  Holbrook \cite{Holbrook1992} disproved this conjecture with a
computer-assisted $3\times3$ example. An explicit,
elementary example, due to G.~Krause, is given in
\cite[Sec.~33]{Bhatia2007Book}.  For this example,
\[
d(\sigma(A),\sigma(B))^2=\frac{28}{13},
\qquad
\lVert A-B\rVert^2=\frac{27}{13},
\]
and hence
\[
\sqrt{\frac{28}{27}}\le c_3.
\]

For $n\ge1$, define
\[
c_n=
\sup_{\substack{A,B\in M_n(\C)\text{ normal}\\A\ne B}}
\frac{d(\sigma(A),\sigma(B))}{\lVert A-B\rVert},
\qquad
c_*:=\sup_{n\ge1}c_n.
\]
Bhatia, Davis, and McIntosh \cite{BhatiaDavisMcIntosh1983} proved the
existence of a finite dimension-free upper bound.  Bhatia, Davis, and
Koosis \cite{BhatiaDavisKoosis1989} reduced the problem to an extremal
problem in Fourier analysis; combined with the numerical estimates of
H\"ormander and Bernhardsson \cite{HormanderBernhardsson1993}, this yields
\[
c_*<2.9038872828.
\]
Recently, Bondarenko, Ortega-Cerd\`a, Radchenko, and Seip
\cite{BondarenkoEtAl2026} characterized the H\"ormander--Bernhardsson
extremal function and developed high-precision methods for computing it.
These works, however, concern the auxiliary Fourier extremal problem and
leave $c_*$ undetermined.  Thus
\[
\sqrt{\frac{28}{27}}\le c_*<2.9038872828,
\]
and the exact values of both $c_*$ and $c_3$ remain open.  We stress that
$c_3$ here denotes the fixed-dimensional matrix constant, not an auxiliary
constant from the Fourier-extremal literature.

For background and the history of the problem, see
\cite{Bhatia2007Book,Bhatia2007History}.  The spectral-subspace mechanism
used below goes back to Bhatia, Davis, and McIntosh
\cite{BhatiaDavisMcIntosh1983}; see also Bhatia and Holbrook
\cite{BhatiaHolbrook1987}.  The proof combines a Hall-duality witness $(I,J)$ for the optimal matching distance with a rectangular singular-value compression.

Set
\[
m_n:=\left\lfloor\frac{n+1}{2}\right\rfloor.
\]
The main result is as follows.

\begin{theorem}\label{thm:main}
For any two normal matrices $A,B\in M_n(\C)$,
\begin{equation}\label{eq:main}
d(\sigma(A),\sigma(B))^2
\le \sum_{k=1}^{m_n}s_k(A-B)^2.
\end{equation}
Equivalently,
\begin{equation}\label{eq:main-kf}
d(\sigma(A),\sigma(B))
\le \KFnorm{A-B}{m_n}.
\end{equation}
In particular,
\begin{equation}\label{eq:main-op}
d(\sigma(A),\sigma(B))
\le \sqrt{m_n}\,\lVert A-B\rVert.
\end{equation}
\end{theorem}

Consequently,
\begin{equation}\label{eq:cn}
c_n\le
\min\left\{\sqrt{\left\lfloor\frac{n+1}{2}\right\rfloor},\,2.9038872828\right\}.
\end{equation}
For $3\le n\le16$, the square-root term is smaller than the classical
dimension-free upper bound.  In particular,
\[
\sqrt{\frac{28}{27}}\le c_3\le\sqrt2.
\]
For $n=2$, Theorem~\ref{thm:main} recovers the classical equality $c_2=1$.
Because $\sqrt{m_n}$ grows with $n$, the theorem gives no new
dimension-free upper bound for $c_*$; for $n\ge17$, the classical bound in
\eqref{eq:cn} is smaller.

The coefficient $1$ in \eqref{eq:main} is sharp in every dimension.  Indeed,
for $t>0$, $A=\diag(t,0,\ldots,0)$, and $B=0$, one has
\[
d(\sigma(A),\sigma(B))^2=t^2
=\sum_{k=1}^{m_n}s_k(A-B)^2.
\]
Since
\[
\sum_{k=1}^{m_n}s_k(A-B)^2
\le\sum_{k=1}^{n}s_k(A-B)^2
=\lVert A-B\rVert_F^2,
\]
Theorem~\ref{thm:main} refines the Hoffman--Wielandt bound for the
optimal matching distance and also yields a rank-sensitive estimate.

\begin{corollary}\label{cor:rank}
For normal $A,B\in M_n(\C)$,
\[
d(\sigma(A),\sigma(B))
\le
\sqrt{\min\{\operatorname{rank}(A-B),m_n\}}\,
\lVert A-B\rVert.
\]
In particular, if $A-B$ has rank at most one, then
\[
d(\sigma(A),\sigma(B))\le\lVert A-B\rVert.
\]
\end{corollary}

\begin{proof}
Only $\operatorname{rank}(A-B)$ singular values of $A-B$ are nonzero, and
each is at most $\lVert A-B\rVert$.  The assertion follows from
Theorem~\ref{thm:main}.
\end{proof}

The next observation shows that at least two singular values are necessary
in every dimension $n\ge3$.  It also settles the truncation length in the two
smallest nontrivial dimensions.

\begin{proposition}\label{prop:two-terms}
For every $n\ge3$, there exist normal matrices $A,B\in M_n(\C)$ such that
\[
d(\sigma(A),\sigma(B))>s_1(A-B)=\lVert A-B\rVert.
\]
Consequently, among inequalities with coefficient $1$ of the form
\[
d(\sigma(A),\sigma(B))^2
\le\sum_{k=1}^{r_n}s_k(A-B)^2,
\]
one must have $r_n\ge2$ for every $n\ge3$.  In particular, the truncation
length $m_n$ in Theorem~\ref{thm:main} is optimal for $n=3$ and $n=4$.
\end{proposition}

\begin{proof}
Let $(A_0,B_0)$ be a Holbrook example in $M_3(\C)$
\cite{Holbrook1992}, and put
\[
d_0:=d(\sigma(A_0),\sigma(B_0))>\lVert A_0-B_0\rVert.
\]
For $n>3$, choose $\gamma_4,\ldots,\gamma_n\in\C$ mutually distant
and far from $\sigma(A_0)\cup\sigma(B_0)$, so that
\[
|\gamma_j-\gamma_k|>d_0\quad(j\ne k),
\qquad
|\gamma_j-z|>d_0
\quad
(z\in\sigma(A_0)\cup\sigma(B_0)).
\]
Set
\[
A=A_0\oplus\diag(\gamma_4,\ldots,\gamma_n),\qquad
B=B_0\oplus\diag(\gamma_4,\ldots,\gamma_n).
\]
The matching that pairs each $\gamma_j$ with itself and restricts to
an optimal matching of $\sigma(A_0),\sigma(B_0)$ has maximum distance
$d_0$, so $d(\sigma(A),\sigma(B))\le d_0$.  Conversely, in any matching
of maximum distance strictly less than $d_0$, every $\gamma_j$ must be
paired with itself, since all other possible partners are at distance
greater than $d_0$ from it.  Restricting such a matching to the
remaining three eigenvalues then yields a matching between
$\sigma(A_0)$ and $\sigma(B_0)$ of maximum distance strictly less than
$d_0$, which contradicts the definition of $d_0$.  Hence
\[
d(\sigma(A),\sigma(B))=d_0,
\qquad
\lVert A-B\rVert=\lVert A_0-B_0\rVert.
\]
For $n=3$ we simply take $(A,B)=(A_0,B_0)$.  Since $m_3=m_4=2$ and at
least two terms are necessary, the truncation length in
Theorem~\ref{thm:main} is optimal for $n=3,4$.
\end{proof}

It remains to prove Theorem~\ref{thm:main}.  The proof rests on the
following duality for the optimal matching distance.  To the best of
our knowledge, neither \eqref{eq:main} nor its consequence
$c_3\le\sqrt2$ has appeared previously.

\medskip
\noindent\textbf{Use of AI.}
GPT-5.6 sol was used to obtain the main results of this paper. The author
subsequently reviewed the entire manuscript in full and takes full
responsibility for its correctness.

\section{Proof of the main theorem}\label{sec:proof}

\begin{lemma}[Hall duality]\label{lem:hall}
For indexed multisets
$\mathcal A=\{\alpha_1,\ldots,\alpha_n\}$ and
$\mathcal B=\{\beta_1,\ldots,\beta_n\}$ in $\C$, let
\[
d(\mathcal A,\mathcal B)
=\min_{\tau\in S_n}\max_{1\le i\le n}
  |\alpha_i-\beta_{\tau(i)}|
\]
be their optimal matching distance.  Then
\begin{equation}\label{eq:hall}
d(\mathcal A,\mathcal B)
=
\max_{\substack{I,J\subseteq\{1,\ldots,n\}\\|I|+|J|=n+1}}
\min_{\substack{i\in I\\j\in J}}|\alpha_i-\beta_j|.
\end{equation}
\end{lemma}

\begin{proof}
Fix $I,J$ with $|I|+|J|=n+1$ and let $\tau\in S_n$.  Since
$|J^c|=n-|J|=|I|-1$, the image $\tau(I)$, which has $|I|$ elements,
cannot be contained in $J^c$, which has only $|I|-1$ elements.  Hence
some $i\in I$ satisfies $\tau(i)\in J$, and for this $i$
\[
\max_{1\le i\le n}|\alpha_i-\beta_{\tau(i)}|
\ge|\alpha_i-\beta_{\tau(i)}|
\ge\min_{\substack{i\in I\\j\in J}}|\alpha_i-\beta_j|.
\]
The rightmost quantity is independent of $\tau$, so minimizing the
leftmost term over $\tau$ yields
$d(\mathcal A,\mathcal B)\ge\min_{i\in I,j\in J}|\alpha_i-\beta_j|$.
Maximizing over all admissible pairs $I,J$ proves the ``$\ge$''
direction of \eqref{eq:hall}.

For the reverse inequality, put $\delta=d(\mathcal A,\mathcal B)$.  If
$\delta=0$, there is nothing to prove.  Suppose that $\delta>0$, fix
$0\le r<\delta$, and form the bipartite graph in which $\alpha_i$ is joined
to $\beta_j$ precisely when $|\alpha_i-\beta_j|\le r$.  A
perfect matching in this graph would give a matching of
$\mathcal A,\mathcal B$ with maximum distance at most $r<\delta$,
contradicting the definition of $\delta$; hence no perfect matching
exists.  By Hall's marriage theorem \cite{Hall1935}, there is a subset
$I\subseteq\{1,\ldots,n\}$ whose neighborhood $N_r(I)$
satisfies $|N_r(I)|<|I|$.  Put
$J=\{1,\ldots,n\}\setminus N_r(I)$.  Then
\[
|I|+|J|=|I|+n-|N_r(I)|\ge n+1,
\]
and $|\alpha_i-\beta_j|>r$ for every $i\in I$ and $j\in J$.  Passing to
subsets of $I$ and $J$ if necessary, we may assume $|I|+|J|=n+1$, and
the cross distances remain greater than $r$.  Consequently the
right-hand side of \eqref{eq:hall} is at least $r$.  Since this holds
for every $r<\delta$, it is at least $\delta=d(\mathcal A,\mathcal B)$.
This proves the reverse inequality, and the lemma follows.
\end{proof}

\begin{proof}[Proof of Theorem~\ref{thm:main}]
Write
\[
\delta=d(\sigma(A),\sigma(B)),\qquad M=A-B.
\]
Choose $I,J$ attaining the maximum in \eqref{eq:hall}, and put
$p=|I|$, $q=|J|$.  Then $p+q=n+1$, and
\begin{equation}\label{eq:cross-distance}
|\alpha_i-\beta_j|\ge\delta
\qquad(i\in I,\ j\in J).
\end{equation}

Fix complete orthonormal eigenbases of $A$ and $B$ once and for all
(repeated eigenvalues cause no difficulty), and let $U_I\in\C^{n\times p}$
and $V_J\in\C^{n\times q}$ be the isometries whose columns are the
eigenvectors labeled by $I$ and $J$, respectively. The column spaces of $U_I$ and $V_J$ have dimensions $p$ and $q$,
respectively, and $p+q=n+1>n$; hence they cannot be disjoint.  Choose a unit vector $z$ in
the intersection and write
\[
z=U_Ia=V_Jb;
\]
then $a=U_I^*z$ and $b=V_J^*z$ are unit vectors in $\C^p$ and $\C^q$.  For
$X:=U_I^*V_J\in\C^{p\times q}$, we then have $Xb=a$.  Since $U_I$ and
$V_J$ are isometries, $\lVert X\rVert\le1$.  On the other hand, since $\lVert b\rVert=1$,
$\lVert X\rVert\ge\lVert Xb\rVert=\lVert a\rVert=1$.  Hence
\begin{equation}\label{eq:X}
s_1(X)=\lVert X\rVert=1,
\qquad
\lVert X\rVert_F^2\ge1.
\end{equation}
Let
\[
D_I=\diag(\alpha_i)_{i\in I},\qquad
D_J=\diag(\beta_j)_{j\in J}.
\]
Normality gives $U_I^*A=D_IU_I^*$ and $BV_J=V_JD_J$.  Consequently,
\begin{equation}\label{eq:sylvester}
C:=D_IX-XD_J=U_I^*(A-B)V_J=U_I^*MV_J.
\end{equation}
Entrywise,
\[
C_{ij}=(\alpha_i-\beta_j)X_{ij}.
\]
It follows from \eqref{eq:cross-distance} and \eqref{eq:X} that
\begin{equation}\label{eq:C}
\begin{split}
\lVert C\rVert_F^2
&=\sum_{\substack{i\in I\\j\in J}}
  |\alpha_i-\beta_j|^2|X_{ij}|^2 \\
&\ge\delta^2\lVert X\rVert_F^2
\ge\delta^2.
\end{split}
\end{equation}
Let $r_0=\min(p,q)$.  The min--max characterization of singular values, or
equivalently the standard product inequalities for singular values, gives
\begin{equation}\label{eq:compression}
s_k(U_I^*MV_J)\le s_k(M),
\qquad 1\le k\le r_0;
\end{equation}
see, for example, \cite[Chapter~III, Section~5]{Bhatia1997}.
Therefore
\begin{equation}\label{eq:S}
\lVert C\rVert_F^2
=\sum_{k=1}^{r_0}s_k(C)^2
\le\sum_{k=1}^{r_0}s_k(M)^2
\le\sum_{k=1}^{m_n}s_k(M)^2,
\end{equation}
because $p+q=n+1$ implies
\[
r_0=\min(p,q)\le\left\lfloor\frac{n+1}{2}\right\rfloor=m_n.
\]
Combining \eqref{eq:C} and \eqref{eq:S} proves \eqref{eq:main}.  Finally,
$s_k(M)\le s_1(M)=\lVert M\rVert$ gives \eqref{eq:main-op}.
\end{proof}

\section{Open problems}

The preceding results leave the following questions open.

\begin{enumerate}
\item Determine the exact value of the universal sharp constant $c_*$.
  The currently available bounds recalled above are
  \[
  \sqrt{\frac{28}{27}}\le c_*<2.9038872828.
  \]

\item Determine the exact value of $c_3$.  Theorem~\ref{thm:main} gives
  \[
  \sqrt{\frac{28}{27}}\le c_3\le\sqrt2,
  \]
  and it is not known whether the Krause example is extremal.

\item Determine the smallest admissible truncation length in
  \eqref{eq:main} for $n\ge5$.  Proposition~\ref{prop:two-terms} shows that
  at least two terms are necessary for every $n\ge3$ and that the value
  $m_n=2$ is optimal for $n=3,4$; it remains open whether
  $m_n=\lfloor(n+1)/2\rfloor$ is optimal in higher dimensions.

\item Determine the exact values of $c_n$ for $n\ge4$, and decide whether
  \eqref{eq:main-op} is sharp for any $n\ge3$.

\item Characterize equality in \eqref{eq:main} beyond the commuting
  rank-one examples, and determine conditions guaranteeing strict inequality.
\end{enumerate}


\begin{thebibliography}{99}

\bibitem{Bhatia1997}
R.~Bhatia,
\emph{Matrix Analysis},
Graduate Texts in Mathematics 169, Springer, New York, 1997.

\bibitem{Bhatia2007Book}
R.~Bhatia,
\emph{Perturbation Bounds for Matrix Eigenvalues},
Classics in Applied Mathematics 53, SIAM, Philadelphia, 2007.

\bibitem{Bhatia2007History}
R.~Bhatia,
Spectral variation, normal matrices, and Finsler geometry,
\emph{Math. Intelligencer} \textbf{29} (2007), 41--46.

\bibitem{BhatiaDavisKoosis1989}
R.~Bhatia, C.~Davis, and P.~Koosis,
An extremal problem in Fourier analysis with applications to operator theory,
\emph{J. Funct. Anal.} \textbf{82} (1989), 138--150.

\bibitem{BhatiaDavisMcIntosh1983}
R.~Bhatia, C.~Davis, and A.~McIntosh,
Perturbation of spectral subspaces and solution of linear operator equations,
\emph{Linear Algebra Appl.} \textbf{52/53} (1983), 45--67.

\bibitem{BhatiaHolbrook1987}
R.~Bhatia and J.~A.~R.~Holbrook,
Unitary invariance and spectral variation,
\emph{Linear Algebra Appl.} \textbf{95} (1987), 43--68.

\bibitem{BondarenkoEtAl2026}
A.~Bondarenko, J.~Ortega-Cerd\`a, D.~Radchenko, and K.~Seip,
The H\"ormander--Bernhardsson extremal function,
to appear in \emph{Acta Math.}, arXiv:2504.05205v2, 2026.

\bibitem{Hall1935}
P.~Hall,
On representatives of subsets,
\emph{J. London Math. Soc.} \textbf{10} (1935), 26--30.

\bibitem{HoffmanWielandt1953}
A.~J.~Hoffman and H.~W.~Wielandt,
The variation of the spectrum of a normal matrix,
\emph{Duke Math. J.} \textbf{20} (1953), 37--39.

\bibitem{Holbrook1992}
J.~A.~R.~Holbrook,
Spectral variation of normal matrices,
\emph{Linear Algebra Appl.} \textbf{174} (1992), 131--144.

\bibitem{HormanderBernhardsson1993}
L.~H\"ormander and B.~Bernhardsson,
An extension of Bohr's inequality,
in J.-L.~Lions and C.~Baiocchi (eds.),
\emph{Boundary Value Problems for Partial Differential Equations and
Applications}, RMA Res. Notes Appl. Math. 29, Masson, Paris, 1993,
179--194.

\bibitem{Mirsky1960}
L.~Mirsky,
Symmetric gauge functions and unitarily invariant norms,
\emph{Quart. J. Math. Oxford Ser. (2)} \textbf{11} (1960), 50--59.



\end{thebibliography}
\end{document}